\documentclass{elsarticle}
\usepackage{lineno, hyperref}
\modulolinenumbers[5]
\usepackage{mathtools}
\DeclarePairedDelimiter{\ceil}{\lceil}{\rceil}

\newcommand\round[1]{\left[#1\right]}
\usepackage{natbib}
\usepackage{fullpage}
\usepackage{setspace}
\usepackage{amssymb}
\usepackage{amsmath}
\usepackage[hang]{subfigure} 
\usepackage{tikz}

\usepackage[noend]{algpseudocode}
\usepackage{graphics}

\newtheorem{Defn}{Definition}[section]
\newtheorem{Lemma}{Lemma}[section]
\newtheorem{Theorem}{Theorem}[section]

\newtheorem{Corollary}[Lemma]{Corollary}

\newtheorem{Remark}{Remark}[section]

\newproof{proof}{Proof}
\newproof{Pot}{Proof of Theorem}
\usepackage{algorithm}

\makeatletter
\def\ps@pprintTitle{%
 \let\@oddhead\@empty
 \let\@evenhead\@empty
 \def\@oddfoot{}%
 \let\@evenfoot\@oddfoot}
\makeatother 
\begin{document}

\begin{frontmatter}
\title{Efficient Structural Descriptor Sequence to Identify Graph Isomorphism and Graph Automorphism} 

\author[siva]{Sivakumar Karunakaran} 
\ead{sivakumar\_karunakaranm@srmuniv.edu.in}
\author[lavanya]{Lavanya Selvaganesh\corref{cor1}} 
\ead{lavanyas.mat@iitbhu.ac.in}
\cortext[cor1]{Corresponding author. This work was presented as an invited talk at ICDM 2019, 4-6 Jan 2019 at Amrita University, Coimbatore, India}

\address[siva]{SRM Research Institute, S R M Institute of Science and Technology Kattankulathur, Chennai - 603203, INDIA}

\address[lavanya]{Department of Mathematical Sciences, Indian Institute of Technology (BHU), Varanasi-221005, INDIA}

\begin{abstract}
In this paper, we study the graph isomorphism and graph automorphism problems. We propose a novel technique to analyze graph isomorphism and graph automorphism. Further we handled some strongly regular datasets for prove the efficiency of our technique. The neighbourhood matrix  $ \mathcal{NM}(G) $ was proposed in \cite {ALPaper} as a novel representation of graphs and was defined using the neighbourhood sets of the vertices. It was also shown that the matrix exhibits a bijection between the product of two well known graph matrices, namely the adjacency matrix and the Laplacian matrix. Further, in a recent work\cite{NM_SPath}, we introduced the sequence of matrices representing the powers of $\mathcal{NM}(G)$ and denoted it as $ \mathcal{NM}^{\{l\}}, 1\leq l \leq k(G)$ where $ k(G) $ is called the \textbf{iteration number}, $k(G)=\ceil*{\log_{2}diameter(G)} $. 
In this article we introduce a structural descriptor given by a sequence and clique sequence for any undirected unweighted simple graphs with help of the sequences of matrices $ NM^{\{l\}} $. The $ i^{th} $ element of structural descriptor sequence encodes the complete structural information of the graph from the vertex $ i\in V(G) $. The $ i^{th} $ element of clique sequence encodes the Maximal cliques on $ i $ vertices. The above sequences is shown to be a graph invariants and is used to study the graph isomorphism and automorphism problem. 
\end{abstract}

\end{frontmatter}
\section{Introduction} 
One of the classical and long standing problem in graph theory is the graph isomorphism problem. Study of graph isomorphisms is not only of theoretical interest, but also has found applications in diversified fields such as cryptography, image processing, chemical structural analysis, biometric analysis, mathematical chemistry, bioinformatics, gene network analysis to name a few. The challenging task in solving graph isomorphism problem is in finding the permutation matrix, if it exists, that can be associated with the adjacency matrices of the given graph. We also know that there are $ n! $ such permutation matrices that are possible, where $ n $ is the number of vertices. However, finding one suitable choice is time consuming. There are various graph invariants that can be used to identify two non-isomorphic graphs. That is, for any two graphs, if a graph invariant is not equal, then it immediately follows that the graphs are not isomorphic. For example, the following invariants have been well studied for this purpose, such as, number of vertices, number of edges, degree sequence of the graph, diameter, colouring number, eigenvalues of the associated graph matrices, etc. However, for each of the invariant, if the invariants are equal one cannot declare the graphs are isomorphic. For example, there are non-isomorphic graphs which have same spectral values and such graphs are referred as isospectral graphs.

There are several algorithms to solve the graph isomorphism problem. Computationally, many algorithms have been proposed for the same. Ullman \cite{Sub_Iso_Ullman} proposed the first algorithm for a more general problem known as subgraph isomorphism whose running time is $ \mathcal{O}(n! n^{3}) $. During the last decade many algorithm such as VF2 \cite{VF2_Luigi}, VF2++\cite{Alpar_VF2++}, QuickSI \cite{QuickSI}, GADDI \cite{GADDI_Algori}, have been proposed to improve the efficiency of Ullman's algorithm. Recent developments were made in 2016 by Babai \cite{Quasipolynomial_Babai}, who proposed a more theoretical algorithm based on divide and conquer strategy to solve subgraph isomorphism problem in Quasi Polynomial time ($ 2^{(\log n)^{\mathcal{O}(3)}} $) which has certain implementation difficulties. Due to the non-availability of polynomial time solvable algorithm many researchers have also studied this problem for special classes of graphs such as trees, planar graphs, interval graphs and bounded parameter such as genus, degree and treewidth graphs \cite{Pol_Fixed_Genus,Poly_tree,Poly_Interval,Poly_Permutation,Poly_Planar}.

Maximal clique and graph automorphism problems are interesting and long standing problems in graph theory. Lots research works are done on graph automorphism problem \cite{REF_1,REF_2,REF_3,REF_4,REF_8,REF_9} and clique problem \cite{REF_10,REF_11,REF_12,REF_13}. 

In this paper we study the graph isomorphism problem and graph automorphism problem using our newly proposed structural descriptor sequence and clique sequence. The above sequences are completely constructed form our novel techniques.  This sequences are used to reduce the running time for classify the non isomorphic graphs and find the automorphism groups. We use the recently proposed matrix representation of graphs using neighbourhood sets called neighbourhood matrix $ \mathcal{NM}(G) $ \cite{ALPaper}. Further, we also iteratively construct a finite sequence of powers of the given graph and their corresponding neighbourhood matrices \cite{NM_SPath}. This sequence of graph matrices help us to define a collection of measures, capturing the structural information of graphs. 

The paper is organized as follows: Section 2 presents all required basic definitions and results. Section 3 is the main section which defines the collection of measures and structural descriptor and the main result for testing graph isomorphism. In Section 4, we discuss the clique sequence and its time complexity. Section 5 ensure the efficiency of our structural descriptor sequence and clique sequence. Section 6, describe the way of finding automorphism groups of a given graph and time complexity. We conclude the paper in section 7. 
\section{Sequence of powers of $ \mathcal{NM }(G)$ matrix} 
Throughout this paper, we consider only undirected, unweighted simple graphs. For all basic notations and definitions of graph theory, we follow the books by J.A. Bondy and U.S.R. Murty \cite{Graphtheory} and D.B. West \cite{GraphTheoryWest}. In this section, we present all the required notations and definitions. 

A graph $ G $ is an ordered pair $ (V(G),E(G)) $ consisting of a set $ V(G) $ of vertices and a set $ E(G) $, disjoint from $ V(G) $, of edges, together with an incidence function $ \psi_{G}  $ that associates with each edge of $ G $ an unordered pair of vertices of $ G $. As we work with matrices, the vertex set $ V(G) $ is considered to be labelled with integers $ \{1,2,...,n\} $. For a vertex $v\in V(G)$, let $N_{G}(v)$ denote the set of all neighbors of $ v $. The degree of a vertex $v$ is given by $ d_{G}(v) $ or $|N_{G}(v)|$. The diameter of the graph is denoted by $ diameter(G)$ and the shortest path distance between two vertices $i$ and $j$ in $G$ is denoted by $ d_{G}(i,j), i,j\in V(G)$. Let $ A_{G} \text{ or } A(G), D(G) $ and $ C(G)(:=D(G)-A(G)) $ denote the adjacency matrix, degree matrix and the  Laplacian/admittance matrix of the grpah $ G $ respectively. Two graphs $ G $ and $ H $ are isomorphic, written $ G \cong H $, if there are bijections $ \theta :V(G)\rightarrow V(H) $ and $ \phi: E(G)\rightarrow E(H) $ such that $ \psi_{G}(e)=uv $ if and only if $ \psi_{H}(\phi(e))=\theta(u)\theta(v); $ (that is a bijection of vertices preserving adjacency) such a pair of mappings is called an isomorphism between $ G  $ and $ H $. An automorphism of a graph is an isomorphism of a graph to itself. In the case of a simple graph, an automorphism is just a permutation $ \alpha $ of its vertex set which preserves adjacency: if $ uv $ is an edge then so is $ \alpha (u) \alpha(v) $. The automorphism groups of $ G $ denoted by $ Aut(G) $, is the set of all automorphisms of a groups $ G $. Graphs in which no two vertices are similar are called asymmetric; these are the graphs which have only the identity permutation as automorphism. 

The graph isomorphism problem tests whether two given graphs are isomorphic or not. In other words, it asks whether there is a one-to-one mapping between the vertices of the graphs preserving the adjacency. 

\begin{Defn}\label{d3} {\normalfont{ \cite{ALPaper}}}
Given a graph $ G $, the neighbourhood matrix, denoted by $ \mathcal{NM}(G) =(\eta_{ij})$ is defined as 
$$\eta_{ij}=\begin{cases}
-|N_{G}(i)|,  & \text{ if }  i=j\\
|N_{G}(j)-N_{G}(i)|,  &   \text{ if }   (i,j)\in E(G)\\
-|N_{G}(i) \cap N_{G}(j)|,  & \text{ if }    (i,j)\notin E(G) \\
\end{cases} $$
\end{Defn}
The following lemma facilitates us with a way of computing the neighbourhood matrix. Proofs of the following results are given in appendix for immediate reference as they are awaiting publication elsewhere. 

\begin{Lemma}\label{prop.2}{\normalfont{ \cite{ALPaper}}}
Given a graph $ G $, the entries of any row of $ \mathcal{NM}(G) $ corresponds to the subgraph with vertices from the first two levels of the  level decomposition of the graph rooted at the given vertex with edges connecting the vertices indifferent levels. $\hfill \square $
\end{Lemma} 

\begin{Remark}\label{Remark.2}{\normalfont{ \cite{ALPaper}}}
The above lemma also reveals following information about the neighbourhood matrix, which justifies its terminology. For any $ i^{th} $ row of $ \mathcal{NM}(G) $,
\begin{enumerate}
\item The diagonal entries are either negative or zero. In particular, if $ \eta_{ii}=-c $, then the degree of the vertex is $ c $ and that there will be exactly $ c $ positive entries in that row. If $ \eta_{ii}=0 $ then the vertex $ i $ is isolated.
\item For some positive integer $ c $, if $ \eta_{ij}=c $ then $ j\in N_{G}(i) $ and that there exists $ (c-1) $ vertices are adjacent to $ 1 $ and at distance $ 2 $ from $ i $ through $ j $.
\item If $ \eta_{ij}=-c $, then the vertex $ d_{G}(i,j)=2 $ and there exists $ c $ paths of length two from vertex $ i $ to $ j $. In other words, there exist $ c $ common neighbors between vertex $ i $ and $ j $.
\item If an entry, $ \eta_{ij}=0 $ then the distance between vertices $ i $ and $ j $ is at least $ 3 $ or the vertices $ i $ and $ j $ lie in different components.$\hfill \square $
\end{enumerate}
\end{Remark}

\begin{Defn}\label{Def.1} \normalfont{\cite{NM_SPath} }
 Given a graph $ G $, let $ G^{\{1\}}=G$ and  $ \mathcal{NM}^{\{1\}} = \mathcal{NM}(G^{\{1\}}) $. For $l>1$, let $G^{\{l\}}$ be the graph constructed from the graph $G^{\{l-1\}}$ as below:  
$ V(G^{\{l\}})=V(G^{\{l-1\}})  $ and
 $ E(G^{\{l\}})=E(G^{\{l-1\}})\cup \{(i,j): \eta_{ij}^{\{l-1\}}<0, i\neq j\} $. The sequence of matrices, denoted by $ \mathcal{NM}^{\{l\}} $ is defined iteratively as $ \mathcal{NM}^{\{l\}}=\mathcal{NM}(G^{\{l\}}) $ and can be constructed by definition \ref{d3}. We refer to this  finite sequence of matrices as \textbf{\textit{Sequence of Powers of $ \mathcal{NM} $}}.  
 \end{Defn} 
 
  \begin{Remark} \normalfont{\cite{NM_SPath} }
The adjacency matrix of $ A(G^{\{l\}}) $ is given by:  
 $$  A(G^{\{l\}}) = (a^{\{l\}}_{ij}) =\begin{cases}
    1,& \text{ if } \eta_{ij}^{\{l-1\}}\neq 0, i\neq j\\
    0,& \text{ Otherwise }
        \end{cases} $$ 
  $\hfill \square $
 \end{Remark}

 \begin{Defn}\label{Def.2}  \normalfont{\cite{NM_SPath} }   
 Let $k$ be the smallest possible integer, such that, $\mathcal{NM}^{\{k\}}$ has either no zero entry or the number of non zero entries of $\mathcal{NM}^{\{k\}} $ and $\mathcal{NM}^{\{k+1\}} $ are the same.  
 The number $k(G):=k$ is called the \textbf{\textit{Iteration Number}} of $G$.
 \end{Defn}

\begin{Remark}\label{Rem.2} \normalfont{\cite{NM_SPath} }
	When $G = K_n$, the complete graph, $\mathcal{NM}(K_n)$ has no zero entries. Hence for $K_n$, the iteration number is $k(K_n) = 1$. Further, for a graph $G$ with diameter 2, $\mathcal{NM}(G)$ has no zero entries, hence iteration number $k(G) = 1$    $\hfill \square $
\end{Remark}

Let the number of non zero entries in $\mathcal{NM}^{\{k\}} $ be denoted by $z$. Note that $z\leq n^2$. 
\begin{Theorem}\label{Thm.4} \normalfont{\cite{NM_SPath} }
A graph $ G $ is connected if and only if $ z = n^{2} $. In addition, the iteration number $k(G)$ of the graph $G\neq K_{n}$ is given by $ k(G)=\ceil*{\log_{2}(diameter(G))} $. (For $ G=K_{n}, k(G)=1 $ and $ \mathcal{NM}^{\{1\}}=\mathcal{NM}(G)=-C(G)$) $\hfill \square $
 \end{Theorem}
\begin{Corollary}\label{Cor.1} \normalfont{\cite{NM_SPath} }
A graph $G$ is disconnected if and only if $ z<n^{2} $. Further, the iteration number $ k(G) $ of $G$ is given by $k(G)=\ceil*{\log_{2}S} $, where $ S $ is the maximum over the diameter of components of $G$. $\hfill \square $
\end{Corollary}

\begin{Defn}\label{IsoDef1}
Let $  N({G^{\{l\}}},i)$ denote the neighbourhood set of a vertex $ i $ in $ G^{\{l\}} $, that is,
\begin{equation}\label{Isoeq.1}
 N({G^{\{l\}}},i) =\{x : \eta_{ix}^{\{l\}}>0\} 
\end{equation} 

Let $ X(G^{\{l\}},i) $ be the set of vertices given by
 \begin{equation}\label{Isoeq.2}
 X(G^{\{l\}},i)=\{y:\eta_{iy}^{\{l\}}<0, i\neq y \} 
 \end{equation}
\end{Defn}
Note that, when $ l=1 $, $ N(G^{\{1\}},i)=N_{G}(i)$. For any $ l>1 $, $  N({G^{\{l\}}},i)= N({G^{\{l-1\}}},i) \cup  X({G^{\{l-1\}}},i)$, for any given $ i\in V(G). $  

\begin{Theorem}\label{Thm.1}\normalfont{\cite{NM_SPath} }
 Let $G$ be a graph on $n$ vertices and $k(G)$ be the iteration number of $G$. For $1 \leq l \leq k(G)$,  the off-diagonal elements of $\mathcal{NM}^{\{l\}}$ can be characterized as follows: For $ 1\leq i\neq j\leq n $
 \begin{enumerate} 
 \item {$\eta^{\{l\}}_{ij}=0$  if and only if  $ d_{G}(i,j)> 2^{l}$} 
 \item {$\eta^{\{l\}}_{ij}>0$  if and only if  $0 < d_{G}(i,j)\leq 2^{l-1}$} 
 \item {$\eta^{\{l\}}_{ij}<0$  if and only if  $2^{l-1}<d_{G}(i,j) \leq 2^l$}  $\hfill \square $
 \end{enumerate}
 \end{Theorem}

By combining definition \ref{d3}, definition \ref{Def.1} and Theorem \ref{Thm.1} we state the following corollaries without proof. 
\begin{Corollary}\label{RRM1} 
	For a given graph $ G $, if $ \eta^{\{l\}}_{ij} >0$, for some $ l\leq k(G) $, then the following conditions are equivalent: 	
	\begin{enumerate}
		\item $ (i,j)\in E(G^{\{p\}}) $, where $ p\geq l $.
		\item $ \eta^{\{l\}}_{ij}=|N(G^{\{l\}},j) -N(G^{\{l\}},i)| $.
		\item $ d_{G}(i,j)\leq 2^{l-1} $ $\hfill \square $
	\end{enumerate}
\end{Corollary}
\begin{Corollary}\label{RRM2} 
For a graph $ G $, if $ \eta^{\{l\}}_{ij}=0, i\neq j $, then the following conditions are equivalent:
	\begin{enumerate}
		\item$ (i,j)\notin E(G^{\{p\}}) $, where $ 1\leq p\leq l+1 $
		\item $ d_{G}(i,j)>2^{l} $ $\hfill \square $
	\end{enumerate}
\end{Corollary} 
\begin{Corollary}\label{RRM3} 
	For a graph $ G $, if $ \eta^{\{l\}}_{ij}<0, i\neq j$, for some $ l,1\leq l\leq k(G) $ then the following conditions are equivalent,
	\begin{enumerate}
		\item $ (i,j)\in E(G^{\{p\}}) $, where $ p\geq l+1 $
		\item $ \eta^{\{l\}}_{ij}=-|N(G^{\{l\}},i)\cap N(G^{\{l\}},j)|$
		\item $ 2^{l-1}<d_{G}(i,j)\leq 2^{l} $ $\hfill \square $
	\end{enumerate} 
\end{Corollary}
\section{Structural descriptor sequence of a graph}
Let $ \mathcal{NM}^{\{l\}}(G)$ be the sequence of powers of $ \mathcal{NM}  $ corresponding to a graph $ G $, where $ 1\leq l\leq k(G), k(G)=\ceil*{\log_{2}diameter(G)} $. In the following, we define a novel collection of measures to quantify the structure of a graph as follows:
\begin{Defn} \label{IsoDef2}
Let $ w_{1}, w_{2}, w_{3}, w_{4}, w_{5}$ and $ w_{6} $  be six distinct irrational numbers. For $ x\in N(G^{\{l\}},i) $, let \begin{equation}\label{Isoeq.3}
 M_{1}(G^{\{l\}},i,x)={\Bigg(\dfrac{\eta_{ix}^{\{l\}}}{w_{1}} \Bigg)} +   {\Bigg(\dfrac{|\eta_{xx}^{\{l\}}|-   \eta_{ix}^{\{l\}} +w_{3}}{w_{2}} \Bigg).} 
\end{equation} 
Consider an ordering of the elements  $\langle x_{1}, x_{2},...\rangle $ of $ N(G^{\{l\}},i) $, such that

 $ M_{1}(G^{\{l\}},i,x_{1})\leq M_{1}(G^{\{l\}},i,x_{2})\leq...\leq M_{1}(G^{\{l\}},i,x_{|N({G^{\{l\}}},i)|}) $. 

For $ y\in X(G^{\{l\}},i) $, let   \begin{equation}\label{Isoeq.4}
\begin{aligned}
 M_{2}(G^{\{l\}},i,y) {} = & {}  {\Bigg(\dfrac{|\eta_{iy}^{\{l\}}|}{w_{4}} \Bigg)}+ {\Bigg(\dfrac{|N({G^{\{l\}}},y)\cap X({G^{\{l\}}},i)| +w_{3}}{w_{5}} \Bigg)}\\ & +  
 {\Bigg(\dfrac{|\eta_{yy}^{\{l\}}|-|N({G^{\{l\}}},y)\cap X({G^{\{l\}}},i)|-|\eta_{iy_{j}}^{\{l\}}|+w_{3}}{w_{6}} \Bigg).} 
\end{aligned} 
\end{equation}\normalsize Consider an ordering of the elements  $ \langle y_{1}, y_{2},...\rangle $ of $ X(G^{\{l\}},i) $, such that

 $ M_{2}(G^{\{l\}},i,y_{1})\leq M_{2}(G^{\{l\}},i,y_{2})\leq...\leq M_{2}(G^{\{l\}},i,y_{|X({G^{\{l\}}},i)|}) $. 
\end{Defn}
Note that by Lemma \ref{prop.2}, the induced subgraph obtained from two level decomposition of $ G^{\{l\}} $ with root $ i $, is given by the vertex set $ [i\cup N(G^{\{l\}},i)\cup X(G^{\{l\}},i) ] $. In the above definition, $ M_{1} $ is the measure computed as weighted sum to count the number of edges that connect vertices in level 1 to level 2 and the number of edges that connect vertices within level 1. That is, for $x\in N(G^{\{l\}},i), \eta^{\{l\}}_{ix} $ represent the number of vertices connected with $ x $ and not belonging to the same level as $ x $, while $ |\eta^{\{l\}}_{xx}|-\eta^{\{l\}}_{ix} $ counts the number of vertices connected to $ x $ and belonging to the same level as $ x $. We use different weights, namely $ \dfrac{1}{w_{1}}, \dfrac{1}{w_{2}}, \ldots $, to distinguish the same.  

Similarly, $ M_{2} $ is the weighted sum to count  the  number of edges that connect vertices in level $ 2 $ with vertices in level $ 1 $, vertices within level $ 2 $ and the vertices in level $ 3 $. That is, for $ y\in X(G^{\{l\}},i),|\eta^{\{l\}}_{iy}| $ represents the number of vertices adjacent to $ y $ from level $ 1 $, $ |N(G^{\{l\}},y)\cap X(G^{\{l\}},i)| $ counts the number of vertices in level $ 2 $ adjacent with $ y $. Note that as the matrix immediately does not reveal the edges present with in this level. Here one neeeds to do extra effort in finding that. Next term, $ |\eta_{yy}^{\{l\}}|-|N({G^{\{l\}}},y)\cap X({G^{\{l\}}},i)|-|\eta_{iy}^{\{l\}}| $ counts the number of vertices adjacent with $ y $ and not belonging to level 1 or level 2. Here again, we use different weights, $ \dfrac{1}{w_{4}},\dfrac{1}{w_{5}}, \dfrac{1}{w_{6}}$, to keep track of these values.  The irrational number $ w_{3} $ is used to keep a count of the number of vertices which are isolated within the level in level decomposition.  By Corollaries \ref{RRM1}, \ref{RRM2} and \ref{RRM3}, the integer coefficients in $ M_{1}(G^{\{l\}},i,x) $ and $ M_{2}(G^{\{l\}},i,y) $ gives the complete information about the induced subgraph $ [i\cup N(G^{\{l\}},i)\cup X(G^{\{l\}},i)] $ along with the volume of edges connecting outside it for any given $ l $.

Let us now define a finite sequence, where each element is associated with a vertex. We call this vector as \textbf{\textit{Structural Descriptor Sequence}}.
\begin{Defn}\label{IsoDef3}
A finite sequence, known as \textbf{\textit{Structural Descriptor Sequence}}, $ R_{G} (i)$, for $, i\in V(G) $, is defined as the weighted sum of the ordered sets of measure $ M_{1}(G^{\{l\}},i,x_{j}) $ and $ M_{2}(G^{\{l\}},i,y_{j}) $ given by,
\begin{equation} \label{Isoeq.5}
R_{G}(i)=\sum\limits_{l=1}^{k} \dfrac{1}{Irr(l)}  \Bigg({{\sum\limits_{j=1}^{|N({G^{\{l\}}},i)|}\dfrac{M_{1}({G^{\{l\}}},i,x_{j})}{Irr(j)}}+ {\sum\limits_{j=1}^{|X({G^{\{l\}}},i)|}\dfrac{M_{2}({G^{\{l\}}},i, y_{j})}{Irr(j)}}} \Bigg )
\end{equation} where 
 $ Irr $ is a finite sequence of irrational numbers of the form  $ \langle \sqrt{2}, \sqrt{3}, \sqrt{5},...\rangle $.
 \end{Defn}
In the above definition note that for every $ i\in V(G) $, $ R_{G}(i) $ captures the complete structural information about the node $ i $ and hence the finite sequence $ \{R_{G}(i)\} $ is a complete descriptor of the given graph. We use the \textbf{\textit{Structural Descriptor Sequence}} explicitly to study the graph isomorphism problem. Rest of this section, will discuss this problem and its feasible solution. 
\begin{Remark}
The following inequalities are immediate from the definition. For any given $ i\in V(G),1\leq l\leq k(G) $ and $ x\in N(G^{\{l\}},i), y\in X(G^{\{l\}},i) $, we have
\begin{enumerate}
\item $  1\leq \eta_{ix_{j}}^{\{l\}}\leq (n-1) $ 
\item  $ 0\leq  |\eta_{x_{j}x_{j}}^{\{l\}}|-   \eta_{ix_{j}}^{\{l\}}\leq (n-2) $
\item $ 1\leq |\eta_{iy_{j}}^{\{l\}}|\leq (n-2)$
\item  $ 0\leq |N({G^{\{l\}}},y_{j})\cap X({G^{\{l\}}},i)|\leq (n-3)  $
\item $ 0\leq \Big(|\eta_{y_{j}y_{j}}^{\{l\}}|-|N({G^{\{l\}}},y_{j})\cap X({G^{\{l\}}},i)|-|\eta_{iy_{j}}^{\{l\}}|\Big)\leq (n-3) $ $\hfill \square $
\end{enumerate} 
\end{Remark}

It is well known that two graphs with different number of vertices or with non-identical degree sequences are non-isomorphic. Hence, we next prove a theorem to test isomorphism between two graphs having same number of vertices and identical degree sequence. 

\begin{Theorem} \label{Poly_Iso_Thm_1}
If two graphs $ G $ and $ H $ are isomorphic then the corresponding sorted structural descriptor sequence $ R_{G} $ and $ R_{H} $ are identical.
\end{Theorem}
\begin{proof}
Suppose $ G $ and $ H $ are isomorphic, then we have $ k(G)=k(H) $, and there exist an adjacency preserving bijection $\phi :V(G)\rightarrow V(H) $, such that $ (i,j)\in E(G)  $ if and only if $ (\phi(i),\phi(j))\in E(H) $ where $ 1\leq i,j\leq n $.

It is well known that, for any $ i\in V(G) $ and for given $ l,  1\leq l\leq k(G) $ the subgraph induced by  $ \{x\in V(G) : d_{G}(i,x)\leq 2^{l-1}\}$ is isomorphic to the subgraph induced by $\{\phi(x)\in V(H) : d_{H}(\phi(i),\phi(x))\leq 2^{l-1}\}  $. Similarly, the subgraph induced by set of vertices $  \{y\in V(G) : 2^{l-1}<d_{G}(i,y)\leq 2^{l}\}$ is isomorphic to $\{\phi(y)\in V(H) :2^{l-1}< d_{H}(\phi(i),\phi(y))\leq 2^{l}\}  $. By Corollary \ref{RRM1} and Corollary \ref{RRM3} and Definition \ref{IsoDef1} we have $ N(G^{\{l\}},i) $  isomorphic to $ N(H^{\{l\}},\phi(i)) $ and $X(G^{\{l\}},i) $ is isomorphic to $ X(H^{\{l\}},\phi(i)) $. By equations (\ref{Isoeq.3}) and (\ref{Isoeq.4})  we have, $\forall i\in V(G)  $ and $ \forall x_{j}\in N(G^{\{l\}},i)$  $M_{1}(G^{\{l\}},i,x_{j})=M_{1} (H^{\{l\}},\phi(i),\phi(x_{j})) $  and for every $ i\in V(G) $, $ y_{j}\in X(G^{\{l\}},i) $ we have $ M_{2}(G^{\{l\}},i,y_{j})=M_{2} (H^{\{l\}},\phi(i),\phi(y_{j})) $.
Since each element of the sequence $ M_{1}(G^{\{l\}},i,x_{j}),  M_{1}(H^{\{l\}},\phi(i),\phi(x_{j})), M_{2}(G^{\{l\}},i,y_{j}) $ and $ M_{2}(H^{\{l\}},\phi(i),\phi(y_{j}) ) $ are linear combinations of irrational numbers, by equating the coefficients of like terms, we further get the following 5 equalities between entries of $ \mathcal{NM}^{\{l\}}(G) $ and $ \mathcal{NM}^{\{l\}}(H) $. 
\begin{eqnarray}\label{ee1}
 \eta^{\{l\}}_{ix_{j}} & = & \eta^{\{l\}}_{\phi(i)\phi(x_{j})} \\
\label{ee2}
 \eta^{\{l\}}_{iy_{j}} & = & \eta^{\{l\}}_{\phi(i)\phi(y_{j})} \\
\label{ee3}
 \big(|\eta^{\{l\}}_{x_{j}x_{j}}|-\eta^{\{l\}}_{ix_{j}}\big )& = &\big ( |\eta^{\{l\}}_{\phi (x_{j})\phi (x_{j})}|-\eta^{\{l\}}_{\phi(i) \phi (x_{j})}\big ) \\
\label{ee4}
 \big |N({G^{\{l\}}},y_{j})\cap X({G^{\{l\}}},i)\big |& = &\big |N({H^{\{l\}}},\phi (y_{j}))\cap X({H^{\{l\}}},\phi(i))\big |
 \end{eqnarray}
\begin{equation} \label{ee5}
\begin{aligned}
 \big(|\eta_{y_{j}y_{j}}^{\{l\}}|-|\eta_{iy_{j}}^{\{l\}}|-|N({G^{\{l\}}},y_{j})\cap X({G^{\{l\}}},i)|\big) {} & =  \big(|\eta_{\phi (y_{j})\phi(y_{j})}^{\{l\}}|  -|\eta_{\phi(i)\phi(y_{j})}^{\{l\}}| \\ &  -|N({H^{\{l\}}},\phi(y_{j}))\cap X({H^{\{l\}}},\phi(i))|\big)  
 \end{aligned}
\end{equation}
\normalsize 
where $ x_{j}\in N(G^{\{l\}},i)  $ and $ y_{j}\in X(G^{\{l\}},i) $, for every $ i\in V(G) $ and for every $ l, 1\leq l\leq k(G) $. 

Therefore there exists a bijection from $ R_{G}$ to $ R_{H} $ satisfying. $ R_{G}(i)=R_{H}(\phi(i)) $, $ \forall i\in V(G)$. 
\end{proof} 

\begin{Remark}\label{REMK3}
The converse of Theorem \ref{Poly_Iso_Thm_1} need not hold. For example in the case of strongly regular graphs with parameters $ (n,r,\lambda,\mu) $, where $ n-$ number of vertices, $ r- $ regular, $ \lambda- $ number of common neighbours of adjacent vertices and $ \mu- $ number of common neighbours of non adjacent vertices, we see that the sequences are always identical for any graphs. 	
\end{Remark}

\subsection{Algorithm to compute Structural descriptor sequence}
Now we present a brief description of proposed algorithm to find the structural descriptor sequence for any given graph $ G $. The pseudocodes of the algorithm is given in Appendix. The objective of this work is to get the unique structural sequence given by a descriptor of graph $ G $ using Algorithm \ref{ALO.1} - \ref{ISOALO.3}. The algorithm is based on the construction of the sequence of graphs and its corresponding matrices $ \mathcal{NM}^{\{l\}},1\leq l\leq k(G) $. The structural descriptor sequence is one time calculation for any graph. The sorted sequences are non identical then we conclude corresponding graphs are non-isomorphic. Moreover polynomial time is enough to construct the structural descriptor sequence. 

Given a graph $ G $, we use the algorithm \ref{ALO.1} from \cite{NM_SPath} to compute the sequence of powers of neighbourhood matrix $ \mathcal{NM}^{\{l\}}$. In this module(Algorithm \ref{ALO.1}), we compute the $ \mathcal{NM}(G) $ matrix and the sequence of $ k(G) $ matrices, namely powers of $ \mathcal{NM} (G)$ [denoted by $ SPG(:,:,:) $] associated with $ G $. The algorithm is referred as  \textsc{SP-$\mathcal{NM}(A)$}, where $ A $ is the adjacency matrix of $G$.

Next, we compute a sequence of real numbers $ M_{1}(G^{\{l\}},i,x_{j})$ and $ M_{2}(G^{\{l\}},i,y_{j}),\forall i,$ and for given $l,1\leq l\leq k(G) $ using the equations (\ref{Isoeq.1}), (\ref{Isoeq.2}), (\ref{Isoeq.3}), (\ref{Isoeq.4}) and (\ref{Isoeq.5}) which corresponds to the entries of $ \mathcal{NM} $ describing the structure of two level decomposition rooted at each vertex $i\in V(G^{\{l\}})$, and for any $l,  1\leq l\leq k(G)  $. For evaluating equations (\ref{Isoeq.3}) and (\ref{Isoeq.4}), we choose the weights $ w_{1} $ to $ w_{7} $ as: $ w_{1}=\sqrt{7}, w_{2}=\sqrt{11}, w_{3}=\sqrt{3}, w_{4}=\sqrt{13}, w_{5}=\sqrt{17}, w_{6}=\sqrt{19}  $, which are all done by \ref{ISOALO.3}.

Atlast, we compute a sequence of $ n $ real numbers given by $ R_{G} $ in equation (\ref{Isoeq.5}) for every vertex of the given graph $ G $, denoted by $ R_{G} $. The sequence $  R_{G} $ is constructed from the structural information of two level decomposition of $ G^{\{l\}}, 1\leq l\leq k$, where $k=\ceil*{\log_{2}diameter(G)} $ using the entries of sequence of powers of $ \mathcal{NM}^{\{l\}} $, which are all done by \ref{ISOALO.2}.

\subsection{Time Complexity}
As stated before, the above described algorithm has been designed in three modules and is presented in Appendix.

 In this module (Algorithm \ref{ISOALO.2}), we compute the structural descriptor sequence $ R_{G} $ corresponding to given graph $ G $. The algorithm is named as \textsc{$S_{-}D _{-} S(A,Irr)$}, where $ A $ is the adjacency matrix of $ G $ and $ Irr- $ square root of first $ (n-1) $ primes. We use Algorithm \ref{ALO.1} in Algorithm \ref{ISOALO.2} to compute the $ \mathcal{NM}(G) $ matrix and the sequence of $ k(G) $ matrices, namely powers of $ \mathcal{NM} (G)$ [denoted by $ SPG(:,:,:) $] associated with $ G $. The algorithm is named as \textsc{SP-$\mathcal{NM}(A)$}, as we apply the product of two matrices namely adjacency A(G) and the Laplacian $C(G)$ to obtain the $\mathcal{NM}(G)$ and we do this for $k(G) $ times, for where $ k(G)= \ceil*{\log_{2}(diameter(G))}$, while $G$ is connected and $ k(G) + 1$ times while $G $ is disconnected. We use Coppersmith - Winograd algorithm \cite{C.W} for matrix multiplication. The running time of Algorithm \ref{ALO.1} is $ \mathcal{O}(k.n^{2.3737}) $.  Finally we use Algorithm \ref{ISOALO.3} in Algorithm \ref{ISOALO.2} to we compute the structural descriptor for each $ \mathcal{NM} $ in sequence of powers of $ \mathcal{NM}(G) $. The algorithm is named as \textsc{Structural$ _{-} $ Descriptor}($ \mathcal{NM},n,Irr $), where $ \mathcal{NM} $ is any sequence of $\mathcal{NM} $ corresponding to given $ G $, $ n- $ number of vertices. The running time of Algorithm \ref{ISOALO.3} is $ \mathcal{O}(n^{3}) $.  Therefore the total running time of Algorithm \ref{ISOALO.2} is $ \mathcal{O}(n^{3}\log n) $.

Note that, the contrapositive of Theorem \ref{Poly_Iso_Thm_1} states that if the sequences $ R_{G} $ and $ R_{H} $ differ atleast in one position, then the graphs are non-isomorphic. We exploit to study the graph isomorphism problem. In Algorithms \ref{ALO.1} - \ref{ISOALO.3}, we implement the computation of structural descriptor sequence and use the theorem to decide if the graphs are non-isomorphic. 
However, inview of Remark \ref{REMK3}, when the sequence are identical for two graphs, we cannot conclude immediately whether the graphs are isomorphic or not. Hence we tried to extract more information about the structure of the graphs. 
In this attempt, by using the first matrix of the sequence namely the $ \mathcal{NM}(G) $, we find Maximal cliques of all possible size. We use this information to further discriminate the given graphs. 

\section{Clique sequence of a graph}
We first enumerate all possible maximal cliques of size $ i, 1\leq i\leq w(G)=$clique number   and store it in a matrix of size $ t_{i}\times i $, where $ t_{i} $ is the number of maximal cliques of size $ i $, denote it by $ LK_{t_{i},i} $. Note that for each $ i $, we obtain a matrix, that is, $ w(G) $ such matrices. Secondly for each $ i $, we count the number of cliques to which each vertex belongs to and store as a vector of size $ n (C_{j}^{i}), 1\leq j\leq n$. In this fashion we get $ w(G) $ such vectors. Consider the sequence obtained by computing for each $ j $, $ R=\Bigg\{\sum\limits_{i=1}^{w(G)}\dfrac{C_{j}^{i}}{Irr(i)}\Bigg\}_{j=1}^{n}$. Finally we construct the clique sequence $ CS $, defined as follows $ CS=\Bigg\{\sum_{j=1}^{n} C_{j}^{i}\cdot R(j)   \Bigg\}_{i=1}^{w(G)} $

\subsection{Description of Maximal cliques Algorithm}
In this section we present a brief description of the proposed algorithm to find the Maximal cliques of all possible size of given graph $ G $. The algorithm based on neighbourhood matrix corresponding to $ G $. 
All possible Maximal cliques work done by Algorithm \ref{Clique_Main} and \ref{Clique_Sub}. Algorithm \ref{Clique_Sub} is iteratively run in Algorithm \ref{Clique_Main} by following inputs,   $ A $ is an adjacency matrix corresponding to given graph,  $ W \subset V $, set of vertices of $ W $ is complete subgraph on $ 3\times d $ vertices where $ d\leftarrow 1,2,... $,   $X\subset V,  X+W\subseteq G $ and vertex labels of vertices in $ X $ is greater than vertex labels in $ W $ and $Y\subset V,  Y+W\subset G $ and vertex labels of vertices in $ X $ is greater than vertex labels in $ Y $. Initially, $W=\emptyset, X=1,2,...,n $ and $ Y=\emptyset $, first we find adjacency matrix $ C $ corresponding to $ [X] $ and the neighourhood matrix $ CL $ corresponding to $ C $. The algorithm runs for each vertices in $ [X] $. $ g_{1} $ is a neighbours of $ i $, if $ |g_{1}| >0$ then $ i $ is not isolated vertex and the neighbours of $ i $ should contained in any $ K_{2},K_{3},... $ to gether with $ i $.  

We define new $ g_{1} $, $ g_{1}=\{q:q>i,q\in g_{1}\} $, this elimination used to reduce the running time and maintain the collection of Maximal cliques. Now we construct the sequence of number of triangles passing through each elements of $ g_{1} $ together with $ i $ and stored in $ p $, this can be done by the entries of neighbourhood matrix in step 7.  Now we find the set of vertices which are all not contained in any $ K_{3} $, $ r=\{f: p(f)=0,f\in g_{1}\} $. If suppose for all vertices in $ Y $ are not adjacent with  $ \{i,r(f)\},f=1,2,...,|r| $ then the vertices $ \{i,r(f)\} $ can add in distinct $ K_{2} $ which is done by step 8 - 12. 

Now we again redefine $ g_{1} $, $ g_{1}=\{q: p(q)>0\} $ and find the edges in upper triangular matrix of $ C(g1,g1) $ and stored in $ a_{4} $. Each edges in $ a_{4} $ together with $ i $ contained in any $ K_{3},K_{4},..,\text{ and  so on,} $.  To find $ s $ is a number of common neighbour between  $ a_{4}(u,1) $ and $ a_{4}(u,2), u=1,2,...,|a_{4}|$. If $ s>1 $ then we find the common neighbours among $ i,a_{4}(u,1), a_{4}(u,2) $ and stored in $ a_{8} $. $ a_{9}=\{q: q>a_{4}(u,2), q\in a_{8}\} $, this set used to reduce the running time and maintain the collection of Maximal cliques. If $ |a_{9}|=1 $ then it possible to add in distinct $ K_{4} $ at the same time it is not possible to add, when there exist an edge between $ a_{8} $ and $ a_{9} $ and atleast one vertex in $ Y $ is adjacent with $ \{i,a_{4}(u,1),a_{4}(u,2), a_{9}\} $. which is done by steps 19 - 23. If suppose $ |a_{9}|=0 $ and $ |a_{8}| =0$ then it possible to add in distinct $ K_{3} $ at the same time it is not possible to add, when atleast one vertex in $ Y $ is adjacent with $ \{i,a_{4}(u,1),a_{4}(u,2)\} $. which is done by steps 25 - 29. At last the set  $a_{9}  $ together with $ \{i,a_{4}(u,1),a_{4}(u,2)\} $, possible to contained in $ K_{4},K_{5},... $ and so on,. 

Upon repeated applications of this iterative process to $ W,X $ and $ Y $ we get all possible Maximal cliques. 
Now we get the following outputs, $ U- $ cell array which contains the number of Maximal cliques contained in each $ i\in V $ and $ Z- $ cell array which contains all possible Maximal cliques.  Our aim is to construct a unique sequence using the outputs $ U $ and $ Z $. Let $ C= $ Let $ C $ is a matrix of size ($ t\times n $), which is constructed from $ U $. Let  $ R=\sum_{i=1}^{t}\dfrac{row_{i}(C)}{Irr(i)}, Irr=\sqrt{2},\sqrt{3},\sqrt{5},... $.  Finally the clique sequence $ CS $ defined as follows, $ CS=\Big\{\sum_{j=1}^{n} {C(i,j)}\cdot{R(j)}    \Big\}_{i=1}^{t} $.
\subsection{Time Compleixty}
As stated before, the above described algorithm has been designed in three modules and is presented in Appendix.

In this module (Algorithm \ref{Clique_Sequence}), we compute the clique sequence and subsequence of clique sequence corresponding to given graph $ G $. The algorithm is named as \textsc{Clique$ _{-} $Sequence}($ A,Irr $), where $ A $ is a adjacency matrix of  $ G $. We run Algorithm \ref{Clique_Main}  in this module which is named as \textsc{Complete$ _{-} $Cliques}($ A $). Algorithm  \ref{Clique_Sub} iteratively run in Algorithm \ref{Clique_Main} which is named as \textsc{Cliques}($ A,W,X,Y $), where  $ A $ is an adjacency matrix corresponding to given graph,  $ W \subset V $, set of vertices of $ W $ is complete subgraph on $ 3d $ vertices where $ d\leftarrow 1,2,... $,   $X\subset V,  X+W\subseteq G $ and vertex labels of vertices in $ X $ is greater than vertex labels in $ W $ and $Y\subset V,  Y+W\subset G $ and vertex labels of vertices in $ X $ is greater than vertex labels in $ Y $. The worst running time of algorithm \ref{Clique_Sub} is $ \mathcal{O}(|X|^{5}) $ and algorihtm \ref{Clique_Main} is $ \mathcal{O}(n^{5}) $ + $ \sum_{i=1}^{ct}\sum_{j=3*i}^{n-2}(n-j)^{5}{j-1 \choose 3*i-1} $, where $ct=\round{\dfrac{n}{4}}$. If we expand the sum of the series that's seems to be very hard. Therefore the worst case running time of the Algorithm \ref{Clique_Sequence} is exponential.

\section{Computation and Analysis}
Idenntifying non-isomorphic graphs along with existing algorithm will be executed as follows:

Given a collection of graphs: $ \mathcal{G} $
\begin{enumerate}
	\item We first compute the Structural descriptor sequence using Algorithm \ref{ALO.1} - \ref{ISOALO.3}. So we have a $ n $ element sequence for each graph in $ \mathcal{G} $. 
	From here, all the distinct sequences computed above represent non-isomorphic graphs ($ \mathcal{G}_{1} $). The remaining graphs $ \mathcal{G}_{2} =(\mathcal{G}-\mathcal{G}_{1}) $ are the input for the next step. 
	\item Here we compute the clique sequences for the graphs in $ \mathcal{G}_{2} $. Here we have a $ n- $ element sequence for each graph in $ \mathcal{G}_{2} $. Now all the distinct sequences represent non-isomorphic graphs. Let the collection be $ \mathcal{G}_{3} $. Let $ \mathcal{G}_{4}=\mathcal{G}_{2}-\mathcal{G}_{3} $
	\item On this collection $ \mathcal{G}_{4} $: we define a relation $ G_{1}\backsim G_{2} $ if and only if $ CS(G_{1})=CS(G_{2}), G_{1},G_{2}\in \mathcal{G}_{4} $. Note that the relation forms an equivalance relation and partitions $ \mathcal{G}_{4} $ into equivalance classes having graphs with identical $ CS $ sequences in the same class. 
	\item We run the existing isomorphism algorithm comparing graphs with in each of te equivalance class. 
\end{enumerate}
Such a preprocessing reduces the computational effort involved in identify in non-isomorphic graphs among a given collection of graphs as compared to running the existing algorithm for all possible pairs in $ \mathcal{G} $. 

Above procedure was implemented to verify our claim on relevant datasets for back of space. We present various existing benchmark datasets.

\underline{\textbf{$ DS_{1} $:}} This dataset contains $ 3854 $ graphs, all graphs are non isomorphic strongly regular on the family of $ (n,r,\lambda,\mu) $, where $ n=35 $, $ r=18 $, $ \lambda=9 $, $ \mu=9 $.

\underline{\textbf{$ DS_{2} $:}} This dataset contains $ 180 $ graphs, all graphs are non isomorphic strongly regular on the family of $ (n,r,\lambda,\mu) $, where $ n=36 $, $ r=14 $, $ \lambda=6 $, $ \mu=4 $.

\underline{\textbf{$ DS_{3} $:}} This dataset contains $ 28 $ graphs, all graphs are non isomorphic strongly regular on the family of $ (n,r,\lambda,\mu) $, where $ n=40 $, $ r=12 $, $ \lambda=2 $, $ \mu=4 $.

\underline{\textbf{$ DS_{4} $:}} This dataset contains $ 78 $ graphs, all graphs are non isomorphic strongly regular on the family of $ (n,r,\lambda,\mu) $, where $ n=45 $, $ r=12 $, $ \lambda=3 $, $ \mu=3 $.

\underline{\textbf{$ DS_{5} $:}} This dataset contains $ 18 $ graphs, all graphs are non isomorphic strongly regular on the family of $ (n,r,\lambda,\mu) $, where $ n=50 $, $ r=21 $, $ \lambda=8 $, $ \mu=9 $.

\underline{\textbf{$ DS_{6} $:}} This dataset contains $ 167 $ graphs, all graphs are non isomorphic strongly regular on the family of $ (n,r,\lambda,\mu) $, where $ n=64 $, $ r=18 $, $ \lambda=2 $, $ \mu=6 $.

\underline{\textbf{$ DS_{7} $:}} This dataset contains $ 1000 $ graphs, all graphs are regular on the family of $ (n,r) $, where $ n=35 $, $ r=18$, in this collection $ 794 $ graphs are non isomorphic. \\

\underline{\textbf{On the dataset $ DS_{1} $}}

Among the $ 3854 $ graphs, we could identify $ 3838 $ graphs to be non-isomorphic with distinct sequences, and remaining $ 16 $ graphs were classified in to $ 8 $ equivalence classes, having $v=\{ 2,2,2,2,2,2,2,2 \}$, the elements of $ v $ are in each of the class. Running the existing algorithm on there ${ v_{i} \choose 2} $ graphs for each $ i $, we can completely distinguish the graph collection which has taken only $ 2376.5268 $ seconds in total, instead of running the existing isomorphism algorithm on $ {3854\choose 2} $  pairs which takes atleast 2 days. 

\underline{\textbf{On the dataset $ DS_{2 } $}}

Among the $180  $ graphs, we could identify $ 81 $ graphs to be non-isomorphic with distinct sequences, and remaining $ 99 $ graphs were classified in to $ 19 $ equivalence classes, having\\
 $v=\{2,2,2,2,2,3,4,4,4,4,5,6,6,7,7,7,7,9,16 \}$, the elements of $ v $ are in each of the class. Running the existing algorithm on there ${ v_{i} \choose 2} $ graphs for each $ i $, we can completely distinguish the graph collection which has taken only $ 18.5414 $ seconds in total. The existing isomorphism algorithm on $ {180 \choose 2} $ pairs takes $ 62.0438 $ seconds.

\underline{\textbf{On the dataset $ DS_{ 3} $}}

Among the $28  $ graphs, we could identify $ 20 $ graphs to be non-isomorphic with distinct sequences, and remaining $ 8 $ graphs were classified in to $ 4 $ equivalence classes, having $v=\{ 2,2,2,2 \}$, the elements of $ v $ are in each of the class. Running the existing algorithm on there ${ v_{i} \choose 2} $ graphs for each $ i $, we can completely distinguish the graph collection which has taken only $ 2.0855 $ seconds in total, the existing isomorphism algorithm on $ {28 \choose 2} $ pairs takes $39.9169 $ seconds. 

\underline{\textbf{On the dataset $ DS_{ 4} $}}

Among the $ 78 $ graphs, we could identify $ 78 $ graphs to be non-isomorphic with distinct sequences, which has taken only $ 5.8620 $ seconds in total, the existing isomorphism algorithm on $ {78 \choose 2} $ pairs takes $ 22.6069$ seconds. 

\underline{\textbf{On the dataset $ DS_{5 } $}}

Among the $ 18 $ graphs, we could identify $ 18 $ graphs to be non-isomorphic with distinct sequences, which has taken only $ 8.4866 $ seconds in total, the existing isomorphism algorithm on $ { 18\choose 2} $ pairs takes $ 8.3194$ seconds. 

\underline{\textbf{On the dataset $ DS_{6 } $}}

Among the $167 $ graphs, we could identify $ 146 $ graphs to be non-isomorphic with distinct sequences, and remaining $ 21 $ graphs were classified in to $ 8 $ equivalence classes, having $v=\{ 2,2,2,2,2,3,3,5 \}$, the elements of $ v $ are in each of the class. Running the existing algorithm on there ${ v_{i} \choose 2} $ graphs for each $ i $, we can completely distinguish the graph collection which has taken only $ 23.8966 $ seconds in total, the existing isomorphism algorithm on $ { 167\choose 2} $ pairs takes $ 12597.2899$ seconds. 

\underline{\textbf{On the dataset $ DS_{ 7} $}}

Among the $1000  $ graphs, we could identify $ 583 $ graphs to be non-isomorphic with distinct sequences, and remaining $ 417 $ graphs were classified in to $ 206 $ equivalence classes, having $v=\{2,2,2,...,(201 \text{ times}), 3,3,3,3,3 \}$, the elements of $ v $ are in each of the class. Running the existing algorithm on there ${ v_{i} \choose 2} $ graphs for each $ i $, we can completely distinguish the graph collection which has taken only $ 508.6248 $ seconds in total, the existing isomorphism algorithm on $ {1000 \choose 2} $ pairs takes $ 108567.0245$ seconds. 

\section{Automorpism Groups of a graph}
In this section we find the automorphism groups of given graph $ G $, using the strucutral descriptor sequence. The above sequence have a complete structural informations for each vertex $ i,i\in V $. Our aim is to get the optimal possibilities for the automorphism groups. From this sequence we can conclude, if any two values of sequence is not identical then the corresponding vertices never be symmetric. By this idea we can do the following, 
\begin{enumerate}
	\item  Let  $ R_{G} $ be the structural descriptor sequence of given graph. 
	\item Let $ fx_{i} =$  Set of vertices in the $ i^{th} $ component of given graph, where $ i=1,2,...,w $ and  $ w $ is a number of components. 
	\item Let $ h= $  Set of unique values in $ R_{G}(fx_{i}) $. 
	\item $v_{q}= \big\{\alpha : R_{G}(\alpha)=h(q), \alpha\in fx_{i}  \big\}, q=1,2,...,|h| $
	\item $ P_{q}= $ Set of all permutations of $ v_{q} $
	\item $ X $ is all possible combinations of $ P_{q}, q=1,2,...,|h| $. Moreover $ X $ is optimal options of automorphism groups. 
	\item Finally we check each possible permutaions in $ X $ with given graph for get a automorphism groups. 
\end{enumerate}
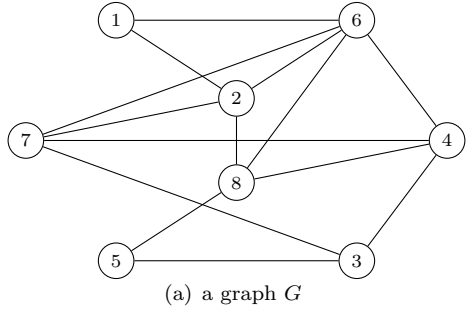
\begin{figure}[H] \scriptsize
	\centering 
	\subfigure[a graph $ G $]{
		\begin{tikzpicture}[scale=0.8,auto=left,every node/.style={draw,circle}] 
		\node (n4) at (5.5,3.5)  {4};
		\node (n3) at (4,1.5)  {3};  
		\node (n1) at (0,5.5)  {1};
		\node (n2) at (2,4.2)  {2};
		\node (n5) at (0,1.5)  {5};
		\node (n6) at (4,5.5)  {6};
		\node (n7) at (-1.5,3.5)  {7};
		\node (n8) at (2,2.8)  {8};
		\foreach \from/\to in {n1/n2, n1/n6, n2/n6, n2/n7, n2/n8, n3/n4, n3/n5, n3/n7, n4/n7, n4/n8, n5/n8, n6/n4,n6/n7,n6/n8}
		\draw (\from) -- (\to)      ; \label{autoF.a}
		\end{tikzpicture}} 
	
	\quad  \subfigure[$ R_{G}- $ {\textit{Structural Descriptor Sequence}} corresponding to $G$]       
	{\begin{tikzpicture} \normalsize \label{autoF.1b}  \node (0,2){  $ R_{G}  = \left( \begin{array}{c c c c c c c c}  
			9.9718  &   8.6746 &   9.4496  &  8.7680 &   9.5123  &  8.3244 &   8.7680 &   8.7649		
			
			\end{array} \right) $}; \end{tikzpicture}}
	
	\caption{A graph $ G $ and its structural descriptor sequence $ R_{G} $}
	\label{AutoFig}
\end{figure}

\begin{Corollary}\label{Auto_Cor}
	If the structural descriptor sequence $ R_{G} $ contains $ n $ distinct elements, then $|Aut( G)|=1 $, that is, $ G $ is an asymmetric graph.
\end{Corollary}
\begin{proof}
	Let the structural descriptor sequence $ R_{G} $ contain $ n $ distinct elements, that is, $ R_{G}(i)\neq R_{G}(j), \forall i\neq j, i,j\in V(G) $. By the definition of $ R_{G} $, it is immediate that for any two vertices $ i $ and $ j $, the two level decompositions of $ G^{\{l\}} $ rooted at $ i $ and $ j $ are not isomorphic for some $ l,1\leq l\leq k(G) $. Since this is true for any two vertices, there exists no non-trivial automorphhism in $ G $. Hence $ Aut(G)$ contains only the identity mapping $ e $. 
\end{proof}

\normalsize
\begin{Remark}
	Note the converse of the above corollary need not hold. For the graph in \normalfont{Figure \ref{autoF.a}}, the structural descriptor sequence $ R_{G} $ is given by Figure \ref{autoF.1b}. Here, $ R_{G} (4)=R_{G}(7)$ but the neighbours of vertex 4 and neighbours of vertex 7 are different as they have different degree sequence and their $ R_{G}-$ values also do not coincide. Hence, we cannot find any automorphism other than identity. 
\end{Remark}
\subsection{Description of Automorphism Group Algorithm}
In this section, we present a brief description of the proposed algorithm to find the all posssible automorphism groups of given graph $ G $, this work done by Algorithm \ref{AutoAl_2} and Algorithm \ref{AutoAl_3}. The pseudocodes of the algorithm is given in Appendix. The objective of this work is find the set of all permutation matrix for the automorphism group. This work  based on the construction of structural descriptor sequenceof given graph $ G $, which is done in Algorithm \ref{ISOALO.2}.  First we find the set of vertices of connected components of given graph. This algorithm run for each connected component of the graph. Let $ fx $ be the vertices of $ i^{th} $ component of the given graph and $ h $ is the set of unique values in $ R_{G}(fx) $. Now we find a vertex set $ v $ which have identical values in $ R_{G}(fx) $ for each values in $ h $, then we find all possible permutations of $ v $. Now we find the combinations of all possible $ v $ corresponding to each elements in $ h $ and stored in $ X $. Finally we check for each permutations in $ X $ with given graph to get a automorphism groups.  

\subsection{Time Complexity}
\begin{enumerate}
\item In this module (Algorithm \ref{AutoAl_2}), we compute automorphism groups of given graph $ G $. The algorithm is nammed as \textsc{$A_{-}M_{-}G(A)$}, where $ A- $ Adjacency matrix $ G $. Algorithm \ref{AutoAl_3} is iteratively run in this module. In Algorithm \ref{AutoAl_3}, we compute optimal options of automorphism groups of given set of vertices in $ G $ this algorithm named as \textsc{$ M_{-}O_{-}A_{-}G(R_{G},h,fx) $}, where $ R_{G} $ structural descriptor sequence, $ fx- $ set of vertices in the connected component and $ h- $ unique values in $ R_{G}(fx) $. The worst case running time of this algorithm is $ \mathcal{O}(n!) $ when the graph is connected and all values are identical in a sequence. In case of $ \mathcal{O}(n!) $ we can reduce more time on the following, $ M $ is set of multiplicities of $ h $ in $ R_{G}(fx) $ which is greater than 1.  $ M=\{m_{1},m_{2},...,m_{d}\}, d\leq |h| $. Let $ y_{1}=(m_{1}! -1) $, $ y_{j}=(y_{j-1}\times (m_{j}!-1))+(y_{j-1}+ (m_{j}!-1)) $. Therefore the complexity of Algorithm \ref{AutoAl_3} is $\sum_{j=1}^{d}y_{j}\leq |fx|! \leq  n!  $. Therefore the worst case running time of Algorithm \ref{AutoAl_2} is $ \mathcal{O}(n^{2}.n!) $. 
\end{enumerate}

\section{Conclusion}
In this paper, we have studied the graph isomorphism problem and graph automorphism problem. We determining the class of non-isomorphic graphs from given graph collection.  We have proposed a novel technique to compute a new feature vector graph invariant called as \textit{Structural Descriptor Sequence}, which encodes complete structural information of the graph with respect to each node. This sequence is also unique in the way it is computed. The proposed \textit{Structural Descriptor Sequence} has been shown useful in identifying non-isomorphic graphs from given collection of graphs. We have proved that if any two sorted \textit{Structural Descriptor Sequence} are not identical then the graphs are not isomorphic. 

Further, we also propose a polynomial time algorithm for test the sufficient part of graph isomorphism between two graphs, which runs in $ \mathcal{O}(kn^{3}) $, where $ k=\ceil*{\log_{2}diameter(G)} \leq \log n$. In this paper we also proposed an algorithm for finding the automorphism groups of given graph, using the structural descriptor sequence.

\section{Acknowledgements}
The authors would like to acknowledge and thank DST-SERB Young Scientist Scheme, India [Grant No. SB/FTP/MS-050/2013] for their support to carry out this research at SRM Research Institute, SRM Institute of Science and Technology during the initial period of the grant.
 Mr. Sivakumar Karunakaran would like to thank SRM Research Institute and the Department of Mathematical Sciences, IIT (BHU) for their initial support during the preparation of this manuscript. 
 Mr. Sivakumar Karunakaran would also like to thank Dr. A. Anuradha, Assistant Professor, Department of Mathematics, SRM Institute of Science and Technology for fruitful discussions during the initial stages.  

\section*{References}

\appendix
\section{Algorithm } 
In this section, we present the $ MATLAB $ pseudocode in the form of algorithm for testing isomorphism between two graphs and discuss their running time. 

\begin{algorithm}[H]   \caption{\textsc{Structural Descriptor Sequence of $ G $.\label{ISOALO.2}}}
	\textbf{Objective:} To Find the {\textit{Structural Descriptor Sequence}} corresponding to the given undirected unweighted simple graph $ G $ on $ n $ vertices using the sequence of powers of $ \mathcal{NM} (G)$.  \\
	\textbf{Input:} $ A- $ Adjacency matrix corresponding to the graph $ G $ and $Irr- $ is a set of $ (n-1)$ irrational numbers, where $Irr= \langle \sqrt{2},\sqrt{3},\sqrt{5},...\rangle ; $\\
	\textbf{Output:} $ R_{G} -$ is a set of $ n $ real numbers corresponding to each vertex of the graph $ G $. 
	\begin{algorithmic}[1] 
		\Procedure{$  R_{G} = S_{-}D_{-}S $  }{$A,Irr$} 
		\State [$ n $ \text{, } $ k $  \text{, } $ SPG $ ] $ \leftarrow $  \textsc{$ SP$-$\mathcal{NM}$}({$A$})
		\For {$ l \leftarrow 1 \text{ to } k $}
		\State $ \mathcal{NM}\leftarrow SPG(:,:,l) $
		\State $ E $\textsc {$ \leftarrow  $ $Structural_{-}Descriptor $  ($\mathcal{NM},n,Irr$)}
		\State $ S(l,:)\leftarrow \dfrac{E}{Irr(l)} $
		\EndFor 
		\State \textbf{end}
		\If {$ k==1 $} $ R_{G}\leftarrow S $
		\Else { $ R_{G} \leftarrow \sum S $ } 
		\EndIf 
		\State \textbf{end}
		\EndProcedure	
	\end{algorithmic} 
	
\end{algorithm}	  	

\begin{algorithm}[H] \cite{NM_SPath} \caption{\textsc{Sequence of powers of $ \mathcal{NM } $.\label{ALO.1}}}
	\textbf{Objective:} To find the iteration number $k $ and construct the sequence of powers of $\mathcal{NM} $ matrix of a given graph $G$, that  is $ \mathcal{NM}^{\{l\}} $, for every  $l,1 \leq l\leq k $. 
	
	\textbf{Input:} Adjacency matrix $ A  $ of a undirected unweighted simple graph $ G $.
	
	\textbf{Output:} The number of vertices $ n $, iteration number $k $ and $SPG-$ A three dimensional matrix of size $(n\times n\times k )$, that is for each $l$, $ 1 \leq l \leq k  $, the $n\times n-\mathcal{NM}^{\{l\}} $ matrix is given by $ SPG(:,:,l) $.
	
	\begin{algorithmic}[1] 
		\Procedure{$ [n \text{, } k  \text{, } SPG ]= $ SP$-\mathcal{NM}$}{$A$}  
		\State $ l\leftarrow 1; $
		\State  $ n \leftarrow  $ Number of rows or columns of matrix $ A  $
		\State $ z \leftarrow 0$ \Comment{Initialize }
		\While{(True)}
		\State $ L\leftarrow  $ Laplacian matrix of $ A  $
		\State $ \mathcal{NM} \leftarrow A \times L$ \Comment{Construct the $ \mathcal{NM}  $ matrix from $ A  $}.
		\State $ SPG(:,:,l)\leftarrow  \mathcal{NM }  $ \Comment{$ SPG $-sequence of powers of $ \mathcal{NM }  $ matrix.}
		\State $ u\leftarrow   nnz(\mathcal{NM } ) $ \Comment{$ nnz $-Number of non zero entries }
		\If {$u==n^{2}$} $k \leftarrow  l$ \textbf{break}
		\ElsIf {$ isequal(u,z)==1 $} $k \leftarrow  l-1 $ \textbf{break}
		\Else  { $ z\leftarrow u $ }
		\State $ A \leftarrow \mathcal{NM} ; A (A \neq 0)\leftarrow 1;$
		\For{$ p\leftarrow 1 \text{ to } n $}  $A (p,p)\leftarrow 0$
		\EndFor 
		\State \textbf{end}
		\State $l\leftarrow  l+1$
		\EndIf 
		\State \textbf{end}
		\EndWhile 
		\State \textbf{end}
		\EndProcedure
	\end{algorithmic} 
	
\end{algorithm}

\begin{algorithm}[H]   \caption{\textsc{Structural descriptor of $ G $.\label{ISOALO.3}}}	
	\textbf{Objective:} To find the sequence of $ n $ real numbers corresponding to given $ \mathcal{NM} $ which is constructed from the information of two level decomposition of each vertex $ i $.  \\
	\textbf{Input:} $ \mathcal{NM}, n  $ and $ Irr- $ Square root of first $ (n-1) $ primes. \\
	\textbf{Output:} $ E- $  Sequence of real numbers corresponding to $ i\in V $. 
	\begin{algorithmic}[1] 
		\Procedure{$  E = Structural_{-}Descriptor $  }{$\mathcal{NM},n,Irr$}  
		\State $ Deg\leftarrow |diag(\mathcal{NM})| $
		\For{$ i\leftarrow 1:n$} $ X\leftarrow \mathcal{NM}(i,:); \text{  } X(i)\leftarrow 0 ;$ 
		\State $a\leftarrow \big(find(X>0)\big );  \text{ } b\leftarrow\big(find(X<0)\big )  $
		\State $ M_{1} \leftarrow  sort\Bigg(\dfrac{X(a)}{\sqrt{7}}+\dfrac{Deg(a)-\mathcal{NM}(i,a)+\sqrt{3}}{\sqrt{11}}\Bigg)$ \Comment{$ sort- $ Sorting increasing order}
		\For{$ f\leftarrow 1 \text{ to } |a| $}  $ S_{1}\leftarrow  \sum\Big( \dfrac{M_{1}(f)}{Irr(f)}\Big)  $  \textbf{ end }
		\EndFor 
		\If{$ |b|== 0 $} $ S_{2}\leftarrow  0 $
		\Else { $ p\leftarrow  \mathcal{NM}(b,b); $ } $ p(p<0)\leftarrow  0; \text{ }  p(p>0)\leftarrow 1 ;\text{ } p\leftarrow  \sum(p) $ \Comment{Column sum of $ p $}
		\State $ M_{2} \leftarrow  sort\Bigg(\dfrac{|X(b)|}{\sqrt{13}}+\dfrac{p+\sqrt{3}}{\sqrt{17}}+\dfrac{Deg(b)-|X(b)|-p+\sqrt{3}}{\sqrt{19}}\Bigg)$ \Comment{$ sort- $ Sorting increasing order}
		\For{$ f\leftarrow 1 \text{ to } |b| $}  $ S_{2}\leftarrow  \sum\Big( \dfrac{M_{2}(f)}{Irr(f)}\Big)  $ \textbf{ end }
		\EndFor
		\EndIf 
		\State \textbf{end}
		\State $ E(i)\leftarrow S_{1}+S_{2} $			
		\EndFor
		\State \textbf{end}
		\EndProcedure
	\end{algorithmic} 
\end{algorithm}

\begin{algorithm}[H]  \caption{\textsc{Clique$ _{-} $Sequence of graph $ G $}} \label{Clique_Sequence}
	\textbf{Objective:} To get a unique sequence for differentiate the graphs using Algorithm \ref{Clique_Main}.  \\
	\textbf{Input:} Adjacency matrix $ A $ corresponding to given graph $ G $.\\
	\textbf{Output:} $ CS- $ Unique sequence of irrational numbers of size $ (1\times t) $, where t is maximum clique number. 
	\begin{algorithmic}[1] 
		\Procedure {$ CS =$Clique$ _{-} $Sequence}{$ A,Irr $} 
		\State $ n\leftarrow $ number of rows in $ A $; $ Z\leftarrow zeros(1,n)$; $ R\leftarrow Z $; $ C\leftarrow Z  $
		\State \textsc{$ [U,Z] \leftarrow $ Complete$ _{-} $Cliques}$ (A) $;  
		\For {$ i\leftarrow 1  $ to $ |U| $} $ C\leftarrow [C;U\{i\}] $ \textbf{end}\EndFor 
		\For {$ i\leftarrow 1 $ to  $ |C| $} $ R\leftarrow R+\dfrac{C(i,:)}{Irr(i)} $; \textbf{end} \EndFor 
		\For {$ i\leftarrow  1$ to $ |C| $} $ CS(i)\leftarrow  \sum_{j=1}^{n} {C(i,j)}\cdot{R(j)} $ \textbf{end}; \EndFor  
		\EndProcedure	
	\end{algorithmic} 	
\end{algorithm}

\begin{algorithm}[H]   \caption{\textsc{All possible Maximal cliques of given graph $ G $}}  \label{Clique_Main}
	\textbf{Objective:} To find all possible distinct complete subgraph of given graph $ G $.  \\
	\textbf{Input:} $ A- $ Adjacency matrix corresponding to given graph $ G $.\\
	\textbf{Output:}  $ U- $ cell array which contains the number of Maximal cliques contained in each $ i\in V $, $ Z- $ cell array which contains all possible Maximal cliques. 
	\begin{algorithmic}[1] 
		\Procedure{$ [U,Z] $= Complete$ _{-} $Cliques}{$ A $}
		\State $ n\leftarrow $Number of rows in $ A $; $ X\leftarrow 1,2,...,n $; $ Y\leftarrow \emptyset  $ ; $ W\leftarrow \emptyset  $ 
		\State $ [a,b,c,d,LK_{1},LK_{2},LK_{3},LK_{4},J] \leftarrow$  \textsc{Cliques}$ (A,W,X,Y) $; $ ct\leftarrow 1 $
		\State $ U\{ct\}\leftarrow [a;b;c;d] $; $ Z\{ct,1\}\leftarrow LK_{1} $; $ Z\{ct,2\}\leftarrow LK_{2} $; $ Z\{ct,3\}\leftarrow LK_{3} $; $ Z\{ct,4\}\leftarrow LK_{4}$
		\If {$ |J|==\emptyset $} \textbf{return}; \textbf{end} \EndIf
		\State   $ ct\leftarrow 2 $
		\While{$ (1) $}
		\If {$ J==\emptyset $}; \textbf{break}; \textbf{end}
		\EndIf 
		\State $ a\leftarrow zeros(1,n) $; $ b\leftarrow a $; $ c\leftarrow a $; $ d\leftarrow a $;  $ LK_{1}\leftarrow \emptyset $; 
		\State $ LK_{2}\leftarrow \emptyset $; $ LK_{3}\leftarrow \emptyset $; $ LK_{4}\leftarrow \emptyset $;  $ F\leftarrow \emptyset $
		\For {$ i\leftarrow 1 $ to number of rows in $ J $} $ W\leftarrow J\{i,1\} $; $ X\leftarrow J\{i,2\} $; $ Y\leftarrow J\{i,3\}-X $
		\State $ [ac,bc,cc,dc,LK_{1}c,LK_{2}c,LK_{3}c,LK_{4}c,Jc]\leftarrow $  \textsc{Cliques}$ (A,W,X,Y) $; 
		\State  $ a(e)\leftarrow a(e)+ac $; $ b(e)\leftarrow b(e)+bc $; $ c(e)\leftarrow c(e)+cc $; $ d(e)\leftarrow d(e)+dc $;
		\State  $ a(W)\leftarrow a(W)$+Number of rows in $ LK_{1}c $;  $ b(W)\leftarrow b(W)$+Number of rows in $ LK_{2}c$ ;
		\State  $ c(W)\leftarrow c(W)$+Number of rows in $ LK_{3}c $ ; $ d(W)\leftarrow d(W)$+Number of rows in $ LK_{4}c $;
		\State $ LK_{1}\leftarrow [LK_{1}; W \text{ concatenation with } LK_{1}c] $; $ LK_{2}\leftarrow [LK_{2}; W \text{concatenation with } LK_{2}c] $;
		\State $ LK_{3}\leftarrow [LK_{3}; W \text{ concatenation with } LK_{3}c] $; $ LK_{4}\leftarrow [LK_{4}; W \text{concatenation with } LK_{4}c] $;
		\For {$ t\leftarrow 1  $ to Number of rows in $ Jc $}
		\State $ Jc\{t,1\}\leftarrow [W\text{ concatenation with } Jc\{t,1\} ] $;
		\State $ Jc\{t,2\}\leftarrow [Jc\{t,2\} \text{ concatenation with } X] $; $ Q\leftarrow \{k: k+Jc\{t,1\},k\in Y\} $
		\State  $ Jc\{t,3\}\leftarrow [Q,Jc\{t,3\} \text{ concatenation with } X] $
		\EndFor 
		\State \textbf{end}
		\State $ F\leftarrow [F; Jc] $
		\EndFor 
		\State \textbf{end}
		\State $ J\leftarrow F $
		\State $ a\leftarrow a+U\{ct-1\}(4,:); U\{ct-1\}(4,:)\leftarrow \emptyset  $;
		\State $ U\{ct\}\leftarrow [a;b;c;d] $; $ Z\{ct,1\}\leftarrow [Z\{ct-1,4\};LK_{1}] $; $ Z\{ct,2\}\leftarrow LK_{2} $; 
		\State $ Z\{ct,3\}\leftarrow LK_{3} $; $ Z\{ct,4\}\leftarrow LK_{4} $; $ Z\{ct-1,4\}\leftarrow \emptyset  $; 
		\State $ ct\leftarrow ct+1 $;
		\EndWhile
		\State \textbf{end}
		\EndProcedure	
	\end{algorithmic} 	
\end{algorithm}	  	

\begin{algorithm}[H]   \caption{\textsc{Distinct $ K_{1}, K_{2},K_{3} $ and $ K_{4} $ of $[X] $ on given $ G $}} \label{Clique_Sub}
	
	\textbf{Objective:} To find distinct $ K_{1}, K_{2}, K_{3} $ and $ K_{4} $ of $ [X] $ in $ A $.  \\
	\textbf{Input:}  $ A $ is an adjacency matrix corresponding to given graph,  $ W \subset V $, set of vertices of $ W $ is complete subgraph on $ 3d $ vertices where $ d\leftarrow 1,2,... $,   $X\subset V,  X+W\subseteq G $ and vertex labels of vertices in $ X $ is greater than vertex labels in $ W $ and $Y\subset V,  Y+W\subset G $ and vertex labels of vertices in $ X $ is greater than vertex labels in $ Y $.
	
	\textbf{Output:} $ a,b,c,d\in (\mathbb{W})^{1\times |X|} $.  $ a,b,c,d $ - Number of distinct $ K_{1}, K_{2},K_{3}, K_{4} $ contained in $ [X] $ respectively.   $ LK_{1}, LK_{2},LK_{3},LK_{4} $ - Collections of distinct $ K_{1}, K_{2},K_{3},K_{4} $ in $ [X] $ respectively and $ J $ - Collection of possible distinct $ K_{4}, K_{5},... $
	
	\begin{algorithmic}[1]  \small 
		\Procedure{$ [a,b,c,d,LK_{1},LK_{2},LK_{3},LK_{4},J]= $ Cliques}{$ A,W,X,Y $} 
		\State $ C $ is an adjacency matrix corresponding to induced subgraph of $ X $.
		\State $ CL $ is an neighbourhood matrix corresponding to $ C $, $ n\leftarrow |X| $;
		\State  $ at\leftarrow 1 $; $ bt\leftarrow 1 $; $ ct\leftarrow 1 $; $ dt\leftarrow 1 $; $ vt\leftarrow 1 $
		\For {$ i\leftarrow 1 \text{ to } n $} $ g_{1}\leftarrow N_{[X]}(i) $
		\If {$ |g_{1}|>0 $} $ g_{1}\leftarrow g_{1}(g_{1}>i) $
		\If {$ |g_{1}|>0 $}  $p\leftarrow Degree(g1)-CL(i,g1)  $
		\State $ r\leftarrow g_{1}(p==0) $
		\For {$ q\leftarrow 1 \text{ to } |r| $}
		\If {for every $ Y_{w}, w= 1,2,...,|Y| $ is not adjacent with $ \{i,r(f)\} $} 
		\State $ LK_{2}(bt,:)\leftarrow [r(f),i] $; $ bt\leftarrow bt+1 $
		\EndIf 
		\State \textbf{end}
		\EndFor 
		\State \textbf{end} $ g1\leftarrow g1(p>0) $ 
		\If {$ |g_{1}|>1 $} $ a_{4} \leftarrow $  distinct edges in upper triangular matrix of $ C(g1,g1) $.
		\For {$ u\leftarrow 1 $ to $ |a_{4}| $}  $ s\leftarrow |CL(a_{4}(u,2),a_{4}(u,2))|-CL(a_{4}(u,1),a_{4}(u,2)) $
		\If {$ s>1 $} $ a_{8}\leftarrow \{N_{[X]}(i)\cap N_{[X]}(a_{4}(u,1)) \cap N_{[X]}(a_{4}(u,2))\} $
		\State $ a_{9}\leftarrow a_{8}(a_{8}>a_{4}(u,2)) $; $ tr\leftarrow |a_{9}| $
		\If {$ tr==1 $}
		\If {there exist no edge between $ a_{8} $ and $ a_{9} $}
		\If { for every $ Y_{w}, w= 1,2,...,|Y| $ is not adjacent with $\{ i,a_{4}(u,1),a_{4}(u,2),a_{9} \}$ }  
		\State $ LK_{4}(dt,:)\leftarrow [i,a_{4}(u,1),a_{4}(u,2),a_{9}] $; $ dt\leftarrow dt+1 $
		\EndIf 
		\State \textbf{end}
		\EndIf 
		\State \textbf{end}
		\ElsIf {$ tr==0 $}
		\If {$ |a_{8}|==0 $}
		\If {for every $ Y_{w}, w= 1,2,...,|Y| $ is not adjacent with $\{ i,a_{4}(u,1),a_{4}(u,2) \}$ }
		\State  $ LK_{3}(ct,:)\leftarrow [i,a_{4}(u,1),a_{4}(u,2)] $; $ ct\leftarrow ct+1 $
		\EndIf 
		\State \textbf{end}
		\EndIf 
		\State \textbf{end}
		\Else { $ J\{vt,1\}\leftarrow \{i,a_{4}(u,1),a_{4}(u,2)\} $; $ J\{vt,2\}\leftarrow a_{9} $; $ J\{vt,3\}\leftarrow a_{8} $}
		\EndIf
		\State \textbf{end}
		\Else { }
		\If { for every $ Y_{w}, w= 1,2,...,|Y| $ is not adjacent with $\{ i,a_{4}(u,1),a_{4}(u,2) \}$ } 
		\State  $ LK_{3}(ct,:)\leftarrow [i,a_{4}(u,1),a_{4}(u,2)] $; $ ct\leftarrow ct+1 $
		\EndIf 
		\State \textbf{end}
		\EndIf  
		\State \textbf{end}
		\EndFor 
		\State \textbf{end}
		\EndIf 
		\State \textbf{end}
		\EndIf 
		\State \textbf{end}
		\Else { }
		\If { forevry $ Y_{w}, w= 1,2,...,|Y| $ is not adjacent with $\{ i\}$ } 
		\State  $ LK_{1}(at,:)\leftarrow [i] $; $ at\leftarrow at+1 $
		\EndIf 
		\State \textbf{end}		
		\EndIf 
		\State \textbf{end}
		\EndFor 
		\State \textbf{end}
		\EndProcedure	
	\end{algorithmic} 	
\end{algorithm}	  	
\normalsize 

\begin{algorithm}[H]   \caption{\textsc{Automorphism Groups of given graph $ G $}} \label{AutoAl_2}
	\textbf{Objective:} To find the automorphism groups of given undirected unweighted simple graph.  
	
	\textbf{Input:} Graph $ G $, with adjacency matrix $ A $.
	
	\textbf{Output:} $ A_{-}G $- all the automorphisms of $ G $. 
	\begin{algorithmic}[1] 
		\Procedure{$ A_{-}G= A_{-}M_{-}G $}{$ A $}   
		\State $ Irr\leftarrow  $ square root of first $ (n-1) $ primes. 
		\State $ R_{G} \leftarrow S_{-}D_{-}S(A,Irr)  $; \textsc{$ [CS,R] \leftarrow $Clique$ _{-} $Sequence}($ A,Irr $)
		\State $ R_{G}\leftarrow R_{G}+R $; $ Tb\leftarrow  $ set of vertices of connected components
		\For {$ i\leftarrow 1 \text{ to } |Tb| $}
		\State $  fx\leftarrow $ $ Tb\{i\} $;  $Id\leftarrow fx; $;   
		\State $ h\leftarrow$ unique values in $ (R_{G}(fx)) $;  $ asd\leftarrow 1 $
		\If {$ |h|\neq |fx| $} 
		\State { $eL\leftarrow  \text{Sorted edge list corresponding to } $} $ [fx] $
		\State  $ X \leftarrow  M_{-}O_{-}A_{-}G(R_{G},h,fx)$ ;
		\For {$ u\leftarrow 1$ to $ |X| $} $  f\leftarrow X\{u\} $; \Comment{ \scriptsize $ f $ is optional autommorphism groups of size $ (zk\times 2), zk\leq 2n $\normalsize }         
		\State  $ tL\leftarrow $ sorted edge list corresponding to given $ f $
		\If {$ eL $ is identical with $ tL $}                        
		\State $ AMor\{asd\}\leftarrow f$ ;  $asd\leftarrow asd+1 $;\Comment{$ AMor\{asd\} $ is an cell array which stores different size of permutation}
		\EndIf
		\State \textbf{end}
		\EndFor 
		\State \textbf{end}		
		\EndIf
		\State \textbf{end}
		\State $  AMor\{asd\}\leftarrow [Id' \text{ } Id'] ; A_{-}G\{i\}\leftarrow AMor$; 
		\EndFor 
		\State \textbf{end}
		\State \textbf{return}
		\EndProcedure	
	\end{algorithmic} 
\end{algorithm}	  	

\begin{algorithm}[H]   \caption{\textsc{Optimal Options of Automorphism groups of given vertex set of $ G $}} \label{AutoAl_3}
	\textbf{Objective:} To find the optimal options of the automorphism groups for given graph.  
	
	\textbf{Input:} $ R_{G}- $ Structural descriptor sequence corresponding to given graph, $ fx- $ Set of vertices in the connected component and $ h- $ unique values of $ R_{G}(fx) $.
	
	\textbf{Output:} $ X- $ cell array which contains the optimal options of automorphism for the given graph. 
	\begin{algorithmic}[1] 
		\Procedure{$ X= M_{-}O_{-}A_{-}G $}{$ R_{G},h,fx $} 
		\State $ jt\leftarrow 1 $;
		\For {$ i\leftarrow 1:|h| $}  $   v\leftarrow  $ Set of vertices which have the identical values in $ R_{G}(fx) $; 
		\If {$ |v|>1 $} ; $ fz\leftarrow $  All possible permutations of $ v $ ;  
		\State $ U\{jt\}\leftarrow fz $; $ jt\leftarrow jt+1 $;
		\EndIf 
		\State \textbf{end}
		\EndFor 
		\State \textbf{end}
		\If {$ |U|==1 $}   $  X\leftarrow U\{:\} $;
		\Else  {  $  ct\leftarrow 1 $ };$   X\leftarrow U\{1\} $; $  Y\leftarrow U\{2\} $;
		\While{(1)} $  bt\leftarrow 1 $;
		\For {$ i\leftarrow 1\text{ to }  |X| $}
		\For {$ j\leftarrow 1\text{ to }|Y| $ } 
		\State $ RX\{bt\}\leftarrow [X\{i\}\text{ ; } Y\{j\}] $; $ bt\leftarrow bt+1 $  \EndFor 
		\State \textbf{end}
		\EndFor 
		\State \textbf{end}
		\State $  X\leftarrow [X \text{, } Y \text{, } RX] $; $  ct\leftarrow ct+1 $;
		\If {$ ct==|U| $} \textbf{ break}; \textbf{ end} \EndIf 
		\State  $ Y\leftarrow U\{ct+1\} $;
		\EndWhile
		\State \textbf{end}
		\EndIf 
		\State \textbf{end}
		\State \textbf{return}
		\EndProcedure	
	\end{algorithmic} 
\end{algorithm}	  	
\end{document}